\documentclass[11pt]{amsart}

\usepackage{geometry}
\usepackage{amsmath,amssymb,color, array, verbatim}
\usepackage{enumerate}

\newtheorem{thm}{Theorem}

\newtheorem{cor}[thm]{Corollary}
\newtheorem{lem}[thm]{Lemma}

\newtheorem{prop}[thm]{Proposition}
\newtheorem{exa}[thm]{Example}

\theoremstyle{definition}
\newtheorem{defn}[thm]{Definition}

\newtheorem{rem}[thm]{Remark}

\numberwithin{equation}{section}

\newcommand{\Z}{{\mathbb Z}}

\begin{document}

\title{Weakly tripotent rings}
\author{S. Breaz, A. C\^\i mpean}

\address{Simion Breaz: "Babe\c s-Bolyai" University, Faculty of Mathematics and Computer Science, Str. Mihail Kog\u alniceanu 1, 400084, Cluj-Napoca, Romania}

\email{bodo@math.ubbcluj.ro}

\address{Simion Breaz: "Babe\c s-Bolyai" University, Faculty of Mathematics and Computer Science, Str. Mihail Kog\u alniceanu 1, 400084, Cluj-Napoca, Romania}

\email{cimpean$\_$andrada@yahoo.com}

\begin{abstract}
We study the class of rings $R$ with the property that for $x\in R$ at least one of the elements $x$ and $1+x$ are tripotent.  
\end{abstract}


\maketitle

\section{Introduction}


An element $x$ of a ring $R$ is called \textsl{tripotent} if $x^3=x$, and a ring $R$ is \textsl{tripotent} if all its elements are tripotent. Hirano and Tominaga proved in \cite[Theorem 1]{HT} that a ring $R$ is tripotent if and only if every element of $R$ is a sum of two commuting idempotents.
The class of these rings was extended in \cite[Section 4]{trip} to the class of rings $R$ such that every element is a sum or a difference of two commuting idempotents. This is a natural approach since similar methods were applied for other classes of rings. For instance, the classes of clean rings or nil-clean rings were extended in \cite{AA} and \cite{BDZ} to the classes of weakly clean, respectively weakly nil-clean in the following ways: a ring $R$ is weakly (nil-)clean if and only if every element of $R$ is a sum or a difference of a unit (nilpotent) element and an idempotent.

Apparently the natural question which asks if we can properly extend in a similar way the classes of rings which satisfy identities which involve tripotents instead of idempotents has a negative answer, since $-u$ is tripotent whenever $u$ is tripotent. However, we want to propose such an extension, starting by the following remarks: 
\begin{enumerate}
 \item \textit{every element of a ring $R$ is a sum or a difference of two (commuting) idempotents if and only if for every element $x\in R$ at least one of the elements $x$ or $1+x$ is a sum of two (commuting) idempotents}; 
\item \textit{a ring $R$ is weakly clean if and only if for every element $x\in R$ at least one of the elements $x$ or $1+x$ is a sum of a unit and an idempotent.} 
\item \textit{a ring $R$ is weakly nil-clean if and only if for every element $x\in R$ at least one of the elements $x$ or $1+x$ is a sum of a nilpotent and an idempotent.} 
\end{enumerate}

Therefore, we will say that a ring $R$ is \textsl{weakly tripotent} if for every element $x\in R$ at least one of the elements $x$ or $1+x$ is tripotent. Two main differences between the properties of tripotent rings and  weakly tripotent rings are the following: while tripotent rings are always commutative and the class of tripotent rings is closed under direct products, these propoerties are not valid for weakly tripotent rings. One of the main result of the present paper is Theorem \ref{wtrp-gen} where we provide a characterization of commutative weakly tripotent rings as subrings of rings of the form $R_0\times R_1\times R_2$, where $R_0$ is a weakly tripotent ring without nontrivial idempotents (these rings are described in Corollary \ref{irreducible}), $R_1$ is a Boolean ring, and $R_2$ is a tripotent ring of characteristic $3$.  Moreover, we will obtain connections between the class of weakly tripotent rings and the rings studied in \cite{trip}.  In particular, it is proved in the end of the paper that the class of rings such that every element is a sum or a difference of two commuting idempotents is strictly contained in the class of weakly tripotent rings (Corollary \ref{inc-sci-wtri}).

\section{Weakly tripotent rings}

Before we will study the connections between these classes, we state some some basic properties of weakly tripotent rings.

\begin{lem}
Let $R$ be a weakly tripotent ring. Then
\begin{enumerate}[{\rm (1)}]
 \item every subring of $R$ is weakly tripotent;
 \item every homomorphic image of $R$ is weakly tripotent;
 \item $24=0$, hence $R$ has a decomposition $R=R_1\times R_2$ such that $R_1$ and $R_2$ are weakly tripotent and $8R_1=3R_2=0$;
 \end{enumerate}
\end{lem}

For the the case $R$ is of characteristic $3$ weakly tripotent rings are in fact tripotent, we can restrict our study to rings on characteristic $2^k$, $k\in\{1,2,3\}$. We refer to \cite{CS} for other extensions of other generalizations of tripotent rings of characteristic $3$.

\begin{lem}
If $R$ is a ring that is weakly tripotent and has characteristic $3$ then $R$ is tripotent, hence it is isomorphic to a subdirect product of a direct product of $\Z_3$.
\end{lem}
\begin{proof}
The identity $(1+x)^3=1+x$ is equivalent to $1+3x+3x^2+x^3=1+x$ and since $3=0$ this is equivalent to $x^3=x.$
%
\end{proof}

We start with the 

\begin{prop}
The following are equivalent for a ring $R$:
\begin{enumerate}[{\rm (a)}]
 \item $R$ is weakly tripotent;
 \item for every element $x\in R$ at least one of the elements $x$ or $1-x$ is tripotent.
\end{enumerate}
\end{prop}

\begin{proof}
(a)$\Rightarrow$(b) Let $x\in R$ such that $x^3\neq x$. It follows that $(-x)^3\neq -x$, hence $(1-x)^3=1-x$. 

(b)$\Rightarrow$(a) is proved in a similar way. 
\end{proof}

\begin{rem}
Let us remark that a similar equivalence is not valid neither for Boolean rings nor for nil-clean conditions. More precisely:
\begin{enumerate}[{\rm (a)}]
\item a ring is Boolean (clean, respectively nil-clean) if and only if for every element $x\in R$ at least one of the elements $x$ or $1-x$ is idempotent (clean, respectively nil-clean);
\item in the ring $\Z_3$ for every element $x\in R$ at least one of the elements $x$ or $1+x$ is idempotent (resp. nil-clean), but $\Z_3$ is not Boolean or nil-clean, while $\Z_{(15)}$ (i.e. the ring of rational numbers with denominator coprime with $15$) is a weakly clean ring which is not clean, \cite{AA}.
\end{enumerate}
\end{rem}

In the following result we describe a connection between the class of weakly tripotent rings and one of the classes studied in \cite{trip}. A second connection will be established in Corollary \ref{inc-sci-wtri}.

\begin{cor}\label{first-incl}
The class of all weakly tripotent ring is strictly contained in the class of rings such that all elements are sums an idempotent and a tripotent that commute. 
\end{cor}

\begin{proof}
The first statement is obvious since for every element $x$ of a weakly tripotent ring such that $x^3\neq x$ we have that $x-1$ is tripotent. To produce an example, we observe that in the ring $\Z_4\times \Z_4$ every element is a sum of a idempotent and a tripotent, but $x=(1,2)$ and  $(1,1)+x$ are not tripotent elements.   
\end{proof}

It is easy to see that the class of weakly tripotent rings is contained in the class studied in \cite[Section 2]{trip} of those rings such that every element is a sum of two commuting tripotents.

The proof of the following proposition is easy.

\begin{prop}\label{subdir-prod}
If a ring $R$ is weakly tripotent then it can be embedded as a subring of a direct product of subdirectly irreducible weakly tripotents rings.
\end{prop}

However, in the case of characteristic $2^k$ a direct product of weakly tripotent ring is not necessarily weakly tripotent. The proof of the following proposition is standard. We refer to the similar property stated for weakly nil-clean rings in \cite[Theorem 1.7]{AA}

\begin{prop}\label{product-w}
 A direct product $\prod_{i\in I}R_i$ of weakly tripotent rings is weakly tripotent if and only if there exists $i_0\in I$ such that $R_{i_0}$ is weakly tripotent and for all $i\in I\setminus \{i_0\}$ the rings $R_i$ are tripotent.
\end{prop}

The following result was proved in \cite[Theorem 3.6]{trip} by using the fact that if every element of a ring $R$ is a sum of an idempotent and a tripotent that commute then $R$ is strongly nil-clean, hence $R/J(R)$ is Boolean and $J(R)$ is nil, \cite[Theorem 5.6]{wnc}. We present for reader's convenience a proof which uses the general procedure described in \cite[Section 12]{Lam-first}.

\begin{thm}\cite[Theorem 3.6]{trip}\label{th-sitc}
Let $R$ be a ring of characteristic $2^k$. The following are equivalent:
\begin{enumerate}[{\rm (1)}]
 \item every element of $R$ is a sum of an idempotent and a tripotent that commute;
 \item $x^4=x^6$ for every $x\in R$;
 \item $R/J(R)$ is Boolean and $U(R)$ is a group of exponent $2$;
 \item $R/J(R)$ is Boolean, and for every $x\in J(R)$ we have $x^2=2x$.
\end{enumerate}

Consequently, if $R$ satisfies the above conditions then $J(R)$ coincides to the set of all nilpotent elements of $R$. 
\end{thm}

\begin{proof}
(1)$\Rightarrow$(2) This is the implication (2)$\Rightarrow$(3) from \cite[Theorem 3.6]{trip}. 

(2)$\Rightarrow$(3) First, we observe that for every ring $R$ which satisfies (2) and every $x\in U(R)$ we have $x^2=1$.

The class of rings which satisfies (2) is closed with respect to subrings and factor rings.  More, precisely, we embed 
$R/J(R)$ as a subring of a direct product of left primitive rings which are factor rings of $R$, \cite[Theorem 12.5]{Lam-first}. We will prove that every left primitive ring which satisfies (2) is isomorphic to a division ring. 

In order to prove this, we first observe that if $K$ is a division ring of characteristic $2$ then for every $n\geq 2$, for the $n\times n$ matrix $A=\left(\begin{array}{ccccc}
1 & 1 & 0 &\ldots & 0 \\
1 & 0 & 0 &\ldots & 0\\
0 & 0 &  0 &\ldots & 0\\
\vdots & \vdots &  \vdots &\cdots & \vdots \\
0 & 0 &  0 &\ldots & 0 
 \end{array}\right)$ with coefficients in $K$ we have $A^4\neq A^6$. Then for every $n\geq 2$ and every division ring $K$ of characteristic $2$, the rings $M_n(K)$ of $n\times n$ matrices over $K$ do not satisfy (2).

Suppose that $R$ is left primitive and satisfies (2). If $R$ is not left artinian then we use \cite[Theorem 11.19]{Lam-first} to conclude that there exists a subring $S$ of $R$ and a division ring (of characteristic $2$) such that there exists an onto ring homomorphism from $S$ to the ring of $M_2(K)$, which is not possible. Therefore there exists a positive integer $n$ and a division ring $K$ such that $R\cong M_n(K)$. This is possible only if $n=1$, hence $K$ is a division ring of characteristic $2$. Moreover, for every $x\in K$ we have $x^2=1$, which is possible only if $K\cong \Z_2$ (if we have a nonzero element $x\neq 1$ in $K$ then the subfield generated by $x$ has $4$ elements, and it follows that $x^2\neq 1$).

Coming back to the general case, it follows that for every ring $R$ which satisfies (2) then ring $R/J(R)$ is Boolean, since it is a subring of a direct product of copies of $\Z_2$. 
\end{proof}

\begin{cor}\label{cor-prop-sset}
If every element of $R$ is a sum of an idempotent and a tripotent that commute then 
\begin{enumerate}[{\rm (i)}]
 \item the ideal $J(R)$ is nil and it contains all nilpotent elements of $R$;
  \item for every $x\in J(R)$ we have $x^2=2x$, $4x=0$, and $x^3=0$ . 
\item $xy=yx$ for all $x,y\in J(R)$;
  \item $2J(R)^2=0$.
  \end{enumerate}
\end{cor}

\begin{proof}
(i) is obvious.

(ii) For every $x\in J(R)$ the elements $1\pm x$ are invertible. Hence $(1\pm x)^2=1$, and it follows that $x^2=2x=-2x$. Hence $0=4x=x^3$.

(iii) Since $1+x$ and $1+y$ are invertible, and the group $U(R)$ is commutative, we have $(1+x)(1+y)=(1+y)(1+x)$, hence $xy=yx$.

(iv) Observe that $2(x+y)=(x+y)^2=x^2+xy+yx+y^2=2x+2y+xy+yx$, hence $xy=-yx$. Therefore, $2xy=0$.
\end{proof}

In the following corollary we have a characterization of weakly tripotent rings without nontrivial idempotents. This will be useful in order to describe commutative weakly tripotent rings, since in this case every subdirectly irreducible ring is without nontrivial idempotents. 

\begin{cor}\label{irreducible}
The following are equivalent for a  ring $R$ without nontrivial idempotents such that $3\in U(R)$:
\begin{enumerate}[{\rm (1)}]
 \item every element of $R$ is a sum of an idempotent and a tripotent that commute;
 \item $R$ is local  such that $R/J(R)\cong \mathbb{Z}_2$, $x^2=1$ for all $x\in U(R)$;
 \item  $R/J(R)\cong \Z_2$, for every $x\in J(R)$ we have $x^2=2x$;
 \item $R$ is weakly tripotent.
 \end{enumerate}

 Consequently, every weakly tripotent ring without nontrivial idempotents is commutative. 
 \end{cor}

\begin{proof}
(1)$\Rightarrow$(2) This follows from Theorem \ref{th-sitc} since every idempotent of $R/J(R)$ can be lifted to an idempotent of $R$.

(2)$\Rightarrow$(3) This can be proved as in Corollary \ref{cor-prop-sset}(ii).

(3)$\Rightarrow$(4) Let $x\in R\setminus J(R)$. Then $1+x\in J(R)$, and it follows that $(1+x)^2=2(1+x)$. This implies that $x^2=1$. If $x\in J(R)$ then $1+x\notin J(R)$, hence $(1+x)^2=1$, hence $1+x$ is tripotent.

For the last statement, it is enough to observe that every element of $R\setminus J(R)$ is of the form $1+x$ with $x\in J(R)$, and to apply Corollary \ref{cor-prop-sset}.

(4)$\Rightarrow$(1) See Corollary \ref{first-incl}. 
%
%
%
\end{proof}

We recall that a ring is \textsl{abelian} if every idempotent of $R$ is central. 

\begin{cor}\label{abelian}
Every abelian ring such that every element of $R$ is a sum of an idempotent and a tripotent that commute is commutative. In particular, every abelian weakly tripotent ring is commutative.
\end{cor}

\begin{proof}
If we have an abelian weakly tripotent ring, we embed the ring as a subdirect product of a family of subdirectly irreducible abelian weakly tripotent rings (Proposition \ref{subdir-prod}). Since every subdirectly irreducible abelian ring has no nontrivial idempotents, the conclusion follows from Corollary \ref{irreducible}.  
\end{proof}


But, let us remark that there exists non-commutative subdirectly irreducible weakly tripotent rings which are not local rings:

\begin{exa}
The ring $T_2(\Z_2)$ of upper triangular matrices over $\Z_2$ is subdirectly irreducible and weakly tripotent.   
\end{exa}

The above corollary allows us to provide nontrivial examples:

\begin{exa}\label{exemple-fi}{\rm 
(1) If $N$ is an abelian group such that $2^kN=0$, with $k\leq 3$, then the idealization of $N$ by $\Z_{2^k}$, i.e. the ring af all matrices $\left(\begin{array}{cc}
x & n\\
0&x                                                                                                                                                                                                             \end{array}\right),$
is weakly tripotent if and only if $2N=0$.

(2) Let $2\leq k\leq 3$ be a positive integer. If  $N=\Z_4$ then the multiplication defined on the abelian group $R=\Z_{2^k}\times N$  by 
$(x,n)\ast (y,m)=(xy, xm+ny+2mn)$ induces a structure of weakly tripotent ring on $\Z_4\times Z_4$.

(3) The ring $R=\Z_4[X,Y]/(X^2,Y^2,XY-2,2X+2Y)$ is a subdirectly irreducible weakly tripotent ring. In order to prove that $R$ is subdirectly irreducible, we observe that if $I$ is an ideal in $R$ such that $X$ or $Y$ is in $I$ then $2\in I$. Moreover, if $2+X\in I$ then $2X=X(2+X)\in I$ and $2Y+2=Y(2+X)\in I$. Therefore, $2=2X+2Y+2\in I$. It follows that the ideal $\{0,2\}$ is the smallest ideal of $R$.   
}\end{exa}

In order to provide a characterization for general weakly tripotent rings, we need the following lemma. 

\begin{lem}\label{idempo}
Let $R$ be a commutative weakly tripotent ring of characteristic $2^k$ ($k\leq 3$), and let $N$ be the nil radical of $R$. Then for every idempotent $e$ in $R$ one of the following properties is true:
\begin{enumerate}[{\rm (a)}]
 \item $en=0$ for all  $n\in N$;
 \item $en=n$ for all $n\in N$.
\end{enumerate}
\end{lem}

\begin{proof} We will denote $f=1-n$. 
If $n$ is a nilpotent element then $(e+n)^3=e+n$ or $(f-n)^3=f-n$. 

Suppose that $(e+n)^3=e+n$. Using the identities $8n=0$, $n^2=2n$ and $n^3=0$ it follows that $(e+n)^3=e+en$, hence $en=n$. This also implies  that $fn=0$.

If $(f-n)^3=f-n$ it follows in the same way that $fn=n$, hence $en=0$.

Now suppose that there exists $n_1$ and $n_2$ nonzero nilpotents such that $en_1=0$ and $fn_2=0$. Since $n_1+n_2$ is nilpotent, we can suppose that $e(n_1+n_2)=0$. But this implies that $en_2=0$. Therefore $n_2=0$, a contradiction, and the proof is complete.  
\end{proof}

We have the following characterization:

\begin{thm}\label{wtrp-gen}
A commutative ring $R$ is weakly tripotent if and only if $R=R'\times R''$ such that:
\begin{enumerate}[{\rm (1)}]
 \item $R''$ is a tripotent ring of characteristic $3$ or $R''=0$;
 \item $R'=0$ or $R'$ can be embedded as a subring of a direct product 
 $R_0\times \left(\prod_{i\in I}R_i\right)$ such that $R_0$ is a weakly tripotent ring without nontrivial idempotents, and all $R_i$ are boolean rings. 
\end{enumerate}
\end{thm}

\begin{proof}
We only have to prove that if $R$ is a weakly tripotent ring of characteristic $2^k$ then it can be embedded in a direct product as in (2). 

Let $J$ be the  Jacobson radical of $R$, and recall that $J$ coincides to the set of all nilpotent elements of $R$.  We fix an ideal $L$ which is maximal with the property $J\cap L=0$. Then $L\cong (L+J)/J$ can be embedded in $R/J$. Therefore all elements of $L$ are idempotents and $2L=0$. 

Let $\overline{e}$ be an idempotent in $R/L$. Then $R/L\cong \overline{e}R/L\times \overline{f}R/L$, where  $\overline{f}=1-\overline{e}$. Since $(J+L)/L\subseteq J(R/L)$ we can suppose by Lemma \ref{idempo} that $(J+L)/L\subseteq
\overline{e}R$. Therefore, if $K$ is an ideal in $R$ such that $L\subseteq K$ and $K/L\cong \overline{f}R/L$ then $K\cap J=0$. By the maximality of $L$ it follows that $K=L$, hence $R/L$ has no nontrivial idempotents.

In order to complete the proof, we proceed as in the proof of Birkhoff's Theorem, \cite[Theorem 12.3]{Lam-first}.
For every $0\neq x\in L$ we choose an ideal $I_x$ which is maximal with the property $x\notin I_x$. By Zorn's Lemma we can choose $I_x$ such that $J\subseteq I_x$. Therefore $R/I_x$ is Boolean, and there exists an injective ring homomorphism $R\to R/L\times \left(\prod_{0\neq x\in L}R/I_x\right)$ which provides a subdirect decomposition for $R$.   
\end{proof}

We complete this theorem with an example of a weakly tripotent ring $R$ as in Theorem \ref{wtrp-gen}(2) such that it has no a direct factor isomorphic to $R_0$. 

\begin{exa}
{\rm Let $R_0$ be a subdirectly irreducible weakly tripotent ring. For every positive integer $k$ we fix a ring $R_k\cong \Z_2$. In the direct product $\prod_{k\geq 0}R_k$ we consider the ring $R$ generated by $1=(1_k)_{k\geq 0}$ and the ideal $J(R_0)\oplus(\oplus_{k\geq 1}R_k)$ of all families of finite support in $\prod_{k\geq 0}R_k$ such that the $0$-th component is in $J(R_0)$. It is easy to see that $R$ has the desired properties.  
}\end{exa}

\begin{rem}
Let us remark that the ring $R_0$ from Theorem \ref{wtrp-gen} is not necessarily subdirectly irreducible. In order to see this, we consider the ring $R$ from Example \ref{exemple-fi}(1) with $N=Z_2\times Z_2$. This ring is not subdirectly irreducible since for every subgroup $S\leq N$ the set $I_S= \{\left(\begin{array}{cc}
0 & n\\
0&0                                                                                                                                                                                                             \end{array}\right)\in R\mid x\in S\}$ is a minimal ideal of $R$. Therefore, if we view $R$ as a subdirect product of a family of subdirectly irreducible rings, at least two factors of this product are not tripotent rings. 
\end{rem}

Using the same proof as in Theorem \ref{wtrp-gen} we obtain:

\begin{prop}\label{sci}
The following are equivalent for a ring $R$:
\begin{enumerate}[{\rm (a)}]
\item every element of $R$ is a sum or a difference of two commuting idempotents; 

 \item $R$ is isomorphic to a subring of a direct product $R_0\times R_2\times R_2$ such that $R_0$ is $0$ or $\Z_4$, $R_1$ is Boolean, and $R_2$ is a subdirect product of a family of copies of $\Z_3$;
 \end{enumerate}
\end{prop}

\begin{proof}
(a)$\Rightarrow$(b) Let $R$ be a ring such that every element is a sum or a difference of two commuting idempotents. From \cite[Lemma 4.2(5)]{trip} it follows that $R$ is abelian, hence $R$ is commutative by Corollary \ref{abelian}. 


By \cite[Theorem 4.4]{trip} we have $R=R_1\times R_2$ such that $R_1/J(R_1)$ is Boolean such that $J(R_1)=0$ or $J(R_1)=\{0,2\}$ and $R_2$ is a tripotent ring, so it is enough to assume that $R/J(R)$ is Boolean such that  $J(R)=\{0,2\}$. In this hypothesis, we observe that for every idempotent $e\in R$ we have $2e=0$ or $2(1-e)=0$. Therefore, we can apply the same technique as in the proof of Theorem \ref{wtrp-gen} to conclude that $R$ is isomorphic to a subring of a direct product of a subdirectly irreducible factor ring of $R$ and a Boolean ring. Since in this case subdirectly irreducible factor rings of $R$ are isomorphic to $\Z_2$ or $\Z_4$, the proof is complete.

(b)$\Rightarrow$(a) This is obvious.
\end{proof}

\begin{cor}\label{inc-sci-wtri}
If every element of a ring $R$ is a sum or a difference of two commuting idempotents then $R$ is weakly tripotent. The converse is not true.
\end{cor} 

\begin{proof}
The first statement follows from Proposition \ref{sci} and Theorem \ref{wtrp-gen}. 

In order to see that the converse is not true, let us observe that $\Z_8$ is weakly tripotent, but not all elements of $R$ are sums or differences of two commuting idempotents.
 \end{proof}




\end{document}